\documentclass[12pt]{article}

\usepackage[all]{xypic}
\usepackage{amsmath}
\usepackage{amsfonts}
\usepackage{amssymb}
\usepackage{amsthm}
\usepackage{color}

\usepackage[top=2cm, bottom=2cm, left=2cm, right=2cm]{geometry}

\theoremstyle{plain}
\newtheorem{thm}{Theorem}[section]

\theoremstyle{plain}

\newtheorem{prp}[thm]{Proposition}

\newtheorem{lmm}[thm]{Lemma}

\newtheorem{conj}[thm]{Conjecture}

\newtheorem{note}[thm]{Note}

\newcommand {\mb}{\mathbb}
\newcommand {\Z}{\mb Z}
\newcommand {\R}{\mb R}
\newcommand {\C}{\mb C}

\newcommand {\colim}{\textrm{colim}\ }

\newcommand {\lra}{\longrightarrow}

\begin{document}

\title{On Freudenthal theorem, Kahn-Priddy Theorem, and Curits conjecture}

\author{Hadi Zare\\
        School of Mathematics, Statistics,
        and Computer Sciences\\
        University of Tehran, Tehran, Iran\\
        \textit{email:hadi.zare} at \textit{ut.ac.ir}}
\date{}

\maketitle

\begin{abstract}
We verify Curtis conjecture on a class of elements of ${_2\pi_*^s}$ that satisfy a certain factorisation property. To be more precise, suppose $f\in{_2\pi_n^s}$ pulls back to $g\in{_2\pi_n^s}P$ through the Kahn-Priddy map $\lambda:QP\to Q_0S^0$ such that $g$ projects nontrivially to an element $g'\in{_2\pi_n^s}P_{t(n)}$ with $h(g')=0$ where $h:{_2\pi_*}QP_k\to H_*QP_k$ is the unstable Hurewicz map, and $t(n)=\lceil n/2\rceil$. Then, mod out by elements of ${_2\pi_*^s}\simeq{_2\pi_*}QS^0$ satisfying this property, the Curtis conjecture on the image of $h:{_2\pi_*}QS^0\to H_*QS^0$ holds.
\end{abstract}

\textbf{AMS subject classification:$55Q45,55P42$}

\tableofcontents

\section{Introduction and statement of results}
Let $QS^0=\colim \Omega^iS^i$ be the infinite loop space associated to the sphere spectrum. Curtis conjecture then reads as follows
(see \cite[Proposition 7.1]{Curtis} and \cite{Wellington} for more discussions).

\begin{conj}[{Curtis Conjecture}]\label{conjecture}
The image of the unstable Hurewicz homomorphism $h:{_2\pi_*^s}\simeq{_2\pi_*}QS^0\lra H_*QS^0$ is given by
$$\Z/2\{h(\eta),h(\nu),h(\sigma),h(\theta_i)\}$$
where $\eta,\nu,\sigma$ are the Hopf invariant one elements and $\theta_i\in{_2\pi_{2^{i+1}-2}^s}$ are Kervaire invariant elements.
\end{conj}

After work of Hill, Hopkins and Ravenel \cite{HHR} we know that the elements $\theta_i$ are not to exist for $i>6$. So, as a consequence the above conjecture implies that the image of $h$ is finite.\\

Here and throughout, we write ${_p\pi_*^s}$ and ${_p\pi_*}$ for the $p$-components of $\pi_*^s$ and $\pi_*$, respectively. We also write $H_*$ for $H_*(-;k)$ where the coefficient ring will be clear from the context with an interest in $k=\Z/p$; although some of our results such as Theorem \ref{main1} holds when $k$ is an arbitrary $p$-local commutative coefficient ring. We work $p$-locally or $p$-complete which again will be clear from the context. \\

\section{Preliminaries}
We wish to examine Conjecture \ref{conjecture} in the light of Freudenthal theorem and Kahn Priddy theorem. Suppose $E$ is a connected and connective spectrum, i.e. $\pi_iE\simeq0$ for $i<1$. Write $\Omega^\infty E=\colim \Omega^iE_i$ for the infinite loop space associated to $E$, $QX$ for $\Omega^\infty(\Sigma^\infty X)$ when $X$ is a space, and $\epsilon:\Sigma^\infty \Omega^\infty E\to E$ for the evaluation map which is the stable adjoint to the identity $\Omega^\infty E\to\Omega^\infty E$. The map $\epsilon$ induces stable homology suspension $H_*\Omega^\infty E\to H_*E$. For a spectrum $E$, we refer to the Hurewicz homomorphism $h^s:\pi_*E\to H_*E$ as the stable Hurewicz homomorphism whereas, in positive degrees, we refer to $h:\pi_*E\simeq\pi_*\Omega^\infty E\to H_*\Omega^\infty E$ as the unstable Hurewicz homomorphism. The following was proved in \cite[Theorem 1.3]{Zare-PEMS}.

\begin{thm}\label{main1}
Suppose $f:S^n\lra QS^0$ is given so that for some spectrum $E$ there is a factorisation $S^n\stackrel{f_E}{\lra}\Omega^\infty E\lra QS^0$. Then, the following statements hold.\\
(i) If $E$ is $r$-connected, $n\leqslant 2r+1$, and $h^s(f_E)=0$, then $h(f)=0$. In particular, if $E$ is a $CW$-spectrum and $f_E$ maps trivially under $p$-local stable Hurewicz map $h^s_{(p)}:{_p\pi_*}E\to H_*(E;\Z/p)$ only for some prime $p$ then $f$ maps trivially under the $p$-local unstable Hurewicz map ${_p\pi_*}QS^0\to H_*QS^0$. A similar statement holds in the $p$-complete setting.\\
(ii) If the above factorisation is induced by a factorisation of a map of spectra
$S^n\stackrel{\alpha}{\to}E\stackrel{c}{\to}S^0$ where $E$ is a finite $CW$-spectrum of dimension $r$, $n\geqslant 2r+1$, and $c_*=0$ then $h(f)=0$. In particular, if $c_*=0$ holds $p$-locally only for some prime $p$ then $f$ maps trivially under the $p$-local unstable Hurewicz map ${_p\pi_*}QS^0\to H_*QS^0$. A similar statement holds in the $p$-complete setting.
\end{thm}

We also recall that as a corollary of the above theorem, we have proved that \cite[Theorem 1.8]{Zare-PEMS}

\begin{thm}\label{main0}
The image of $h:\pi_*^s\simeq{\pi_*}QS^0\to H_*(QS^0;\Z)$ when restricted to the submodule of decomposable elements is given by
$\Z\{h(\eta^2),h(\nu^2),h(\sigma^2)\}$.
\end{thm}

Next, recall that by Kahn-Priddy theorem \cite[Proposition 3.1]{Kahn-Priddy} the map $\lambda:QP\to Q_0S^0$ induces an epimorphism
$$\lambda_*:{_2\pi_*^sP}\simeq{_2\pi_*}QP\to{_2\pi_*}Q_0S^0\simeq{_2\pi_*^s}.$$
Moreover, the map $\lambda$ is induced by an infinite loop extension of a map of suspension spectra $t:P\to S^0$ where $P$ is the infinite dimensional real projective space with its $0$-skeleton as its base point. Note that by dimensional reasons $t_*=0$.

\section{Preparatory lemmata}
We begin with fixing some notation. We use $P$ for the infinite dimensional projective space and $P^n$ for its $n$-skeleton, and for $0<k<n\leqslant+\infty$ we set $P^n_k=P^n/P^{k-1}$ as well as $P^n_n=S^n$. For $f\in{_2\pi_n}Q_0S^0$, we write $g\in{_2\pi_*}QP$ for any element with $\lambda g=f$ where $\lambda$ is the Kahn-Priddy map. It is known that $\lambda$ is a split epimorphism whose inverse is provided by a map $t:Q_0S^0\to QP$ with $t=\Omega j_2$ where $j_2:QS^1\to Q\Sigma P$ is the second stable James-Hopf map \cite[Corollary 2.14]{Kuhnhomology}. For a map of spaces $f:X\to QY$, $QY=\colim \Omega^i\Sigma^i Y$, we write for $f^s:X\to Y$ where we identify a based space $X$ with its suspensions spectrum $\Sigma^\infty X$.\\

For $g\in{_2\pi_n^s}P$, by cellular approximation, we may consider $g$ as $g:S^n\to P^n$. We ask what is the least $k$ so that $g$ also factors as $S^n\to P^k\to P^n$. Given $g^s\in{_2\pi_n^s}P^n$ and $k>0$ we shall write $g^s_k:S^n\to P_k^n$ and by abuse of notation for the composition $S^n\to P^n_k\to P_k$.

For the sake of studying Conjecture \ref{conjecture}, applying Theorem \ref{main1}, provides us with the following lower bound result.

\begin{lmm}\label{lowerbound1}
Let $k>0$. We have the followings.\\
(i) Suppose $f\in{_2\pi_{2k+1}}Q_0S^0$ is given with $h(f)\neq 0$. Then, $g$ does not factor through $QP^k$.\\
(ii) Suppose $f\in{_2\pi_{2k}}Q_0S^0$ is given with $h(f)\neq 0$. Then, $g$ does not factor through $QP^{k-1}$.
\end{lmm}

\begin{proof}
(i) In this case we have a factorisation of $f$ as
$$S^{2k+1}\to QP^k\to Q_0S^0$$
where the map $QP^k\to Q_0S^0$ is induced by applying $\Omega^\infty$ to $P^k\to P\to S^0$ with $t_*=0$. Theorem \ref{main1} implies that $f_*=0$.\\
(ii) This is similar to (i).
\end{proof}

Next, suppose $g^s:S^n\to P$ does not factor through $P^k$. Consider the cofibe sequence $P^k\to P\stackrel{q}{\to} P_{k+1}=P/P^k$ where $q$ is the projection. In this case, the composition $g^s_{k+1}=qg^s:S^{2k+1}\to\Sigma^\infty P_{k+1}$ represents a nontrivial element in ${_2\pi_n^s}P_{k+1}$. Note that $q_*$ is an isomorphism in $H_*(-;\Z/2)$ for $*\geqslant k+1$. We may ask, knowing $h(f)\neq 0$ and $h(g_k)\neq 0$, what can be said about $f$? We have the following.

\begin{thm}\label{main2}
(i) Suppose $f\in{_2\pi_{2k+1}}Q_0S^0$ is given with $h(f)\neq 0$ so that $h(g_{k+1})\neq 0$. Then, $f$ is a Hopf invariant one element.\\
(ii) Suppose $f\in{_2\pi_{2k}}Q_0S^0$ is given with $h(f)\neq 0$ so that $h(g_{k})\neq 0$. Then, $f$ is a Kervaire invariant one element.
\end{thm}

The proof is based on using knowledge on $H_*QX$ and $H_*Q_0S^0$ to which we refer the reader to \cite{CLM}. In particular, once and for all, we fix that $a_i\in H_iP^n$, likewise in $H_iP^n_k$, denotes a generator. We shall write $x_i\in H_iQ_0S^i$ for a generator which satisfies $\lambda_*a_i=x_i$.

\begin{proof}
(i) For $g_{k+1}=(\Omega^\infty q)g$, if $(g_{k+1})_*\neq 0$ then by dimensional reasons
$$h(g_{k+1})=a_{2k+1}$$
where $a_{2k+1}\in H_{2k+1}P_{k+1}$ is a generator. This implies that $h(g)=a_{2k+1}$ modulo $\ker((\Omega q)_*)$. Consequently,
$$h(f)=x_{2k+1}+\textrm{other terms}.$$
It is known that in this case $f$ has to be a Hopf invariant one element.\\
(ii) The effect of the Hurewicz homomorphism on ${_2\pi_2^s}\simeq\Z/2\{\eta^2=\theta_1\}$ is known. So, we assume $k>1$ and $f$ pulls back to an element $g:S^{2n}\to QP$. As vector spaces we have $H_{2k}QP_k\simeq\Z/2\{a_{2k},a_k^2\}$, so we may write
$$h(g_k)=\epsilon_1a_{2k}+\epsilon_2a_k^2$$
for some $\epsilon_i\in\Z/2$. If $\epsilon_1=1$ then an argument similar to the above shows that $f=\lambda g$ has to be a Hopf invariant one element. But, this is impossible for dimensional reasons; A Hopf invariant one element lives in odd dimensions. Therefore, $h(g_k)\neq 0$ implies that $h(g_k)=a_k^2$. Consequently, we have
$$h(g)=a_k^2+\textrm{other terms}\in H_{2k}QP.$$
\textbf{Case with $2k-1\equiv1\mathrm{mod}4$.} In this case, work of Eccles on codimension one immersions \cite[Proposition 4.1]{Eccles-codimension} shows that
$k=2^i-1$ for some $i$ and $f=\lambda g$ has to be a Kervaire invariant one element, i.e. $f=\theta_i$.\\
\textbf{Case with $2k-1\equiv3\mathrm{mod}4$.} Similar to the previous case, this is related to codimension one immersions where by work of Lannes \cite{Lannes-immersions}, confirming Eccles' predictions \cite[Page 2]{Eccles-codimension}, that the only case corresponds to an immersion $M^3\looparrowright \R^4$ which represents an element of $\pi_4QP$. But, we know that ${_2\pi_4^s}\simeq 0$. So, this case cannot arise at all.
\end{proof}

\begin{note}
We could eliminate $k$ being even case by geometric means as well. If $\epsilon_1=0$ then the stable map $g_k^s:S^{2k}\to P_k$ which factors through $P_k^{2k}$ by cellular approximation acts trivially in homology. Hence, the map $g_k^s:S^{2k}\to P_k^{2k}$ factors through $P_{k}^{2k-1}$. By Freudenthal theorem, $P_{k}^{2k-1}$ admits at least one suspension, so $P_{k}^{2k-1}=\Sigma Y_{k-1}^{2k-2}$ for some path connected $CW$-complex $Y_{k-1}^{2k-2}$ with its bottom cell in dimension $k-1$. On the other hand, as $\epsilon_1=0$, hence $h(g_k)=a_k^2$ which together with $P_{k}^{2k-1}=\Sigma Y_{k-1}^{2k-2}$ reads as $h(g_k)=Q^k\Sigma y_{k-1}$. For $\alpha:S^{2k-1}\to QY_{k-1}^{2k-2}$, being the adjoint of $g_k:S^{2k}\to Q\Sigma Y_{k-1}^{2k-2}$, we have
$$h(\alpha)=Q^ky_{k-1}+D$$
where $D$ is a sum of decomposable terms, appearing as being the kernel of homology suspension. Note that $y_{k-1}$ being coming from the bottom cell is primitive, so $Q^{k}y_{k-1}$ is primitive. So, $D$ has to be a primitive, consequently it must be a square in dimension $2k-1$. Hence, $D=0$. The class $Q^{k}y_{k-1}$ has to be $A$-annihilated. Hence, applying $Sq^1_*$ eliminates the cases with $k$ being even, so $k$ is odd.
\end{note}



\section{Some unstable results}
It is possible to provide some conditions for the vanishing of $h(f)$ in terms of unstable homotopy groups of truncated projective spaces. Of course, the conditions are very strong in terms of homotopy groups, and we don't have a complete understanding of these conditions. But, they might be of special interest, and hopefully applicable some day! We proceed as follows. Suppose $f:S^n\to Q_0S^0$ is given with $h(f)\neq 0$. For the pull back $g$, after Theorem \ref{main2}, we have the following observation.

\begin{lmm}\label{unstable-1}
(i) Let $n=2k+1$. For $g$ as above, $g_{k+1}=q_{k+1}g$ pulls back to an element of $\pi_{2k+1}P_{k+1}$. \\
(i) Let $n=2k$. For $g$ as above, if $q_{k}g:S^{2k}\to QP_{k}$ is trivial in homology, then $g_{k}=q_{k}g$ pulls back to an element of $\pi_{2k}P_{k}$.
\end{lmm}

\begin{proof}
(i) This follows from Freudenthal suspension theorem.\\
(ii) By Freudenthal's theorem, there is an epimorphism $\pi_{2k}\Omega\Sigma P_k\to \pi_{2k}QP_k$. Since, we are in the meta-stable range, so we may consider the EHP sequence
$$\pi_{2k}P_k\to\pi_{2k}\Omega\Sigma P_k\to\pi_{2k}\Omega\Sigma(P_k\wedge P_k)\simeq H_{2k}\Omega\Sigma(P_k\wedge P_k).$$
Now, the latter isomorphism on the right is just given by the Hurewicz homomorphism. So, if $(g_k)_*=0$ then it is in the image of the suspension homomorphism $\pi_{2k}P_k\to\pi_{2k}\Omega\Sigma P_{k}$. This completes the proof.
\end{proof}

By the above theorem, the existence of spherical classes in $H_*QS^0$ would give rise to nontrivial elements in $\pi_{2k}P_k$ or $\pi_{2k+1}P_{k+1}$ which survive under the stabilisation. Moreover, by computations of Section \ref{sec:reduction}, if we role out spherical classes in $H_*Q_0S^0$ which arise from Hopf invariant one elements, then we may well assume that $g_{k+1}$ and $g_{k}$ with $(g_k)_*=0$ pull back to some elements of $\pi_{2k+1}P^{2k}_{k+1}$ and $\pi_{2k}P_{2k-1}$, respectively. At first, it looks that Lemma \ref{unstable-1} provides a theoretical tool to show that in dimensions $H_*QS^0$ cannot host a spherical class. The sufficient vanishing condition in this direction might be stated as follows.

\begin{prp}\label{vanishingcondition-1}
After ruling out Hopf invariant one elements, the following statements hold.\\
(i) Suppose the image of stabilisation map $\pi_{2k+1}P_{k+1}^{2k}\to\pi_{2k+1}^sP_{k+1}^{2k}\simeq\pi_{2k+1}QP_{k+1}^{2k}$ is trivial. Then, the image of the unstable Hurewicz map $h:\pi_{2k+1}QS^0\to H_{2k+1}QS^0$ is trivial.\\
(ii) Suppose the image of stabilisation map $\pi_{2k}P_{k}^{2k-1}\to\pi_{2k}^sP_{k}^{2k-1}\simeq\pi_{2k}QP_{k}^{2k-1}$ is trivial. Then, the image of the unstable Hurewicz map $h:\pi_{2k}QS^0\to H_{2k}QS^0$ is trivial.
\end{prp}

We note that apart from small values of $k$ for which spherical classes in $H_*QS^0$ with $*=2k,2k+1$ are completely known, the range in which the above theorem applies is the metastable range where we have the $\mathrm{EHP}$-sequence. This allows to proceed with further computations.\\
\textbf{Case of $\pi_{2k+1}P_{k+1}^{2k}$.} Consider the following commutative homotopy-homology ladder where the first row is the $\mathrm{EHP}$-sequence
{\tiny
$$\xymatrix{
                                                       &                                                &&\pi_{2k+1}^sP_{k+1}^{2k}\\
\pi_{2k+1}\Omega^2\Sigma P_{k+1}^{2k}\ar[r]\ar[d]^-{\simeq} & \pi_{2k+1}\Omega^2\Sigma(P_{k+1}^{2k}\wedge P_{k+1}^{2k})\ar[r]\ar[d]^-{(\simeq)} & \pi_{2k+1}P_{k+1}^{2k}\ar[r]^-E & \pi_{2k+1}\Omega\Sigma P_{k+1}^{2k}\ar[r]^-{H}\ar[u]^-{\simeq} & \pi_{2k+1}\Omega\Sigma(P_{k+1}^{2k}\wedge P_{k+1}^{2k})\simeq 0\\
\pi_{2k+3}\Sigma P_{k+1}^{2k}\ar[r]\ar[d]^-{h}         & \pi_{2k+3}\Sigma(P_{k+1}^{2k}\wedge P_{k+1}^{2k})\ar[d]^-{h(\simeq)}\\
0\simeq H_{2k+3}\Sigma P_{k+1}^{2k}\ar[r]              & H_{2k+3}\Sigma(P_{k+1}^{2k}\wedge P_{k+1}^{2k})\simeq\Z/2}$$}
where the vanishing of $\pi_{2k+1}\Omega\Sigma(P_{k+1}^{2k}\wedge P_{k+1}^{2k})$ as well as $H_{2k+3}\Sigma P_{k+1}^{2k}$ occurs for dimensional reasons. The commutativity of the squares on the left together with the exactness of the $\mathrm{EHP}$ sequence leaves us with the following short exact sequence
$$\xymatrix{
        &            &                                   & \pi_{2k+1}^sP_{k+1}^{2k}\\
0\ar[r] & \Z/2\ar[r] & \pi_{2k+1}P_{k+1}^{2k}\ar[r]^-E   & \pi_{2k+1}\Omega\Sigma P_{k+1}^{2k}\ar[r]^-{H}\ar[u]^-{\simeq} &  0}$$
By similar reasons, for $\Sigma P_k^{2k-1}$ we have the following short exact sequence
$$\xymatrix{
        &            &                                          & \pi_{2k}^sP_{k}^{2k-1}\\
0\ar[r] & \Z/2\ar[r] & \pi_{2k+1}\Sigma P_{k}^{2k-1}\ar[r]^-E   & \pi_{2k+1}\Omega\Sigma^2 P_{k}^{2k-1}\ar[r]^-{H}\ar[u]^-{\simeq} &  0.}$$
The outcome of these computations is that the kernel of the stabilisation is too small it is unlikely that the map, say $\Z/2\to\pi_{2k+1}P_{k+1}^{2k}$, becomes an isomorphism. Hence, it is very unlikely that this would allows to eliminate spherical classes in $H_*Q_0S^0$ using these classes.

\section{More unstable computations}\label{sec:unstable-2}
1. For $fS^n\to QS^0$, by Freudenthal theorem, it factors as $f:S^n\stackrel{f^n}{\to}\Omega^{n+1}S^{n+1}\to QS^0$. By Kahn-Priddy theorem, the mapping $f^{n}$ also lift to $\Omega^{n+1}\Sigma^{n+1}P^n$ through the Kahn-Priddy map $\lambda_{n}:\Omega^{n+1}\Sigma^{n+1}P^n\to\Omega^{n+1}S^{n+1}$ resulting in an appropriate factorisation of $f^n$ which fits into a commutative diagram as
$$\xymatrix{
                                            & \Omega^{n+1}\Sigma^{n+1}P^n\ar[r]\ar[d]^-{\lambda_{n+1}} & QP\ar[d]^-{\lambda}\\
S^n\ar[r]^-{f^n}\ar[ru]^-{g^n}              & \Omega^{n+1}S^{n+1}\ar[r]                                & QS^0
}$$
It follows that there is definitely a further factorisation of $f$ as
$$S^n\to\Omega^{n+1}\Sigma^{n+1}P^n\to QP^n\to QP\to QS^0.$$

2. Consider the homotopy equivalences $\Omega^{n+1}S^{n+1}\to\Omega^{n+1}P^{n+1}$ and $\Omega^{2n+1}S^{2n+1}\to\Omega^{2n+1}\C P^n$. Through these equivalences, Curtis conjecture what Curtis conjecture implies about spherical classes in $H_*\Omega^{n+1}P^{n+1}$? What is special about the collection of spaces $\{\Omega^nP^n:n>0\}$?



\section{Some reduction results}\label{sec:reduction}
According to Theorem \ref{main2}, in order to a complete verification of Curtis conjecture, we must consider situations in which $f\in{_2\pi_n}Q_0S^0$ is given with $h(f)\neq 0$ and $h(g_i)=0$ where $i=\lceil n/2\rceil$. This does not seem a very easy point to start. Instead, we provide some reduction results which reduces the dimension of the projective space. we suggest an inductive argument which really depends on stable homotopy groups of spheres. But, with the aid of the existing knowledge on these groups together with some well known facts on cohmology operations, some reduction results could be obtained.\\

Suppose $h(f)\neq 0$ is given with $g^s:S^n\to P^n$ factoring through $P^{n-i}$ but not $P^{n-i-1}$. In this case, consider the cofibre sequence
$P^{n-i-1}\to P^{n-i}\stackrel{p_i}{\to}S^{n-i}$. Since $g^s$ does not factor through $P^{n-i-1}$ then the composition $p_ig^s:S^n\to S^{n-i}$ is essential, representing a nontrivial element in ${_2\pi_{i}^s}$. There are two cases: one where we can find a map $S^{n-i}\to S^0$ making the following commutative
$$\xymatrix{
                                                 & P^{n-i-1}\ar[d]\\
S^n\ar[r]^-{g^s}\ar[rd]\ar@{.>}[ru]^-{\nexists}   & P^{n-i}\ar[r]^-{\lambda^s}\ar[d]^-{p_i}    & S^0\\
                                                 & S^{n-i}\ar@{.>}[ru]_-{\exists \overline{\lambda^s}}
                     }$$
and the other case when such a map does not exists. First, we deal with case when such a decomposition exists.

\begin{lmm}
Suppose there exists $\overline{\lambda^s}:S^{n-i}\to S^0$ making the above diagram commutative. Then, $f\in\{\eta^2,\nu^2,\sigma^2\}$.
\end{lmm}

\begin{proof}
In this case $f$ is a decomposable element in ${_2\pi_*^s}$ with $f^s=\overline{\lambda^s}p_ig^s$. By \cite[Theorem 1.8]{Zare-PEMS} $f$ lives in ${_2\pi_i^s}$ for $i=2,6,14$ given by $\Z/2\{\eta^2\},\Z/2\{\nu^2\}$, or $\Z/2\{\sigma^2,\kappa\}$, respectively. For $i=14$, $h(f)\neq 0$ implies that $f=\sigma^2+\epsilon\kappa$ for some $\epsilon\in\Z/2$. However, $\kappa$ is not a decomposable element, so $f=\sigma^2$. This verifies our claim.
\end{proof}

Note that, in the case of above proof, it is known $h(\kappa)=0$ \cite{Za-ideal}, so even in the cases with $f=\sigma^2+\kappa$, we do not get any contradiction to Curtis conjecture. According to the above lemma, we focus on the cases in which $f^s$ is not a decomposable element of ${_2\pi_*^s}$. \\

First, we focus on the $\theta_i$ elements. It is known that $\theta_i=\eta^2,\nu^2,\sigma^2$ for $i=1,2,3$, respectively, i.e. for $i=12,3$ the $\theta_i$ elements are decomposable. It is known that if there exists $\theta_i$ a Kervaire invariant one element then $h(\theta_i)\neq 0$. It then follows from Theorem \ref{main0} that for $i>3$, the $\theta_i$ elements are not decomposable. Note that by work of Eccles \cite[Proposition 4.1]{Eccles-codimension} $\theta_i$ exists if and only if for some element $\theta_i'\in{_2\pi_{2^{i+1}-2}^s}P$ with $\theta_i=\lambda\theta_i'$ we have $h(\theta_i')=a_{2^i-1}+2$ modulo other terms. It then implies that for the composition $q_{2^i-1}\theta_i':S^{2^{i+1}-2}\to P\to P_{2^i-1}$ we have $h(q_{2^i-1}\theta_i')=a_{2^i-1}^2$. And, as shown below, we may approximate $\theta_i'$ by a map $S^{2^{i+1}-2}\to P^{2^{i+1}-4}$. We have obvious lower bound for the top dimension as follows.

\begin{lmm}
For $i>3$, if $\theta_i'$ factors as $S^{2^{i+1}-2}\to P^n$ then $n>2^i-1$.
\end{lmm}

\begin{proof}
Suppose there is such a factorisation. In this case the composition $q_{2^i-1}\theta_i':S^{2^{i+1}-2}\to P_{2^i-1}^{2^i-1}=S^{2^i-1}$ satisfies $h(q_{2^i-1}\theta_i')=a_{2^i-1}^2\in H_*QS^{2^i-1}$. Consequently, $q_{2^i-1}\theta_i'\in{_2\pi_{2^i-1}^s}$ is a Hopf invariant one element with $i>3$. This is a contradiction.
\end{proof}

Let's note that for a path connected space $X$, the stable homology suspension $\sigma_*^\infty:H_*QX\to \widetilde{H}_*X$ which is induced by the evaluation map $\Sigma^\infty QX\to X$, given $g:S^{n}\to QX$ we have $\sigma_*^\infty h(g)=h(g^s)$.

\textbf{Case $i=0$.}
Suppose $g^s$ does not factor through $P^{n-1}$. In this case, the composition $p_0g^s:S^n\to S^n$ is nontrivial (mod $2$). Since, the elements of ${_2\pi_0^s}$ are detected by degree mod $2$, this implies that $p_0g^s$ is nontrivial in homology, hence $h(g^s)=a_n$ for some $n>0$, implying that $h(g)=a_n$ modulo other terms. As noted before, this is equivalent to $f=\lambda g$ being a Hopf invariant one element. So, $n=1,3,7$.\\

The following cases, Case 1, Case 2, Case 3, can be dealt with by using the available description of the $E_2$ term of the ASS for $P$
\cite[Theorems 1.1 and 1.3]{CohenLinMahowald}.\\

\textbf{Case $i=1$.} Suppose $h(g^s)=0$. For $g^s:S^n\to P^n$ the composition $p_ng^s:S^n\to S^n$ is trivial in homology. Since $\pi_nS^n\simeq H_nS^n$, hence $p_ng$ is null, so $g^s$ factor through $P^{n-1}$. We write $g^s:S^n\to P^{n-1}$ for the resulting map. Suppose $g^s$ does not factor through $P^{n-2}$. In this case, the composition $p_{n-1}g:S^n\to S^{n-1}$ represents a nontrivial element in ${_2\pi_1^s}\simeq\Z/2\{\eta\}$. Since $h(\lambda g)=h(f)\neq 0$, hence $\lambda^sg^s$ is essential, where the composition $p_1g^s$ is detected by $h_1$ in the ASS. The class $f^s=\lambda^sg^s$ can only correspond to $h_1h_i$ of the families of permanent cycles in the ASS with $i\geqslant 3$ listed in \cite[Theorem 1.3]{CohenLinMahowald}, or to one of the classes $h_1h_i$ with $i=1,2,3$ \cite[Page 3]{CohenLinMahowald}. In the former, $f^s=\eta_i$, an element of Mahowald family \cite{Ma}, for which it is known that $h(\eta_i)=0$ for all $i\geqslant 3$ \cite[Theorem 1.4]{Zare-PEMS}, contradicting $h(f)\neq 0$. So this case cannot arise. For the latter cases corresponding to $h_1h_i$ with $i=1,2,3$ we have $f=\eta^2,\eta\nu=0,\eta\sigma$. It is known that $h(\eta\sigma)=0$ by \cite[Theorem 1.8]{Zare-PEMS}, contradicting $h(f)\neq 0$. So, we cannot have this choice either. We are then only left with $f=\eta^2$ for which it is known that $h(\eta^2)\neq 0$ with $\eta^2$ being a Kervaire invariant one element. Consequently, in this case, if $h(f)\neq 0$ then $f=\eta^2$.\\
\textbf{Case $i=2$.} Suppose $g^s$ pulls back to $P^{n-2}$ but not to $P^{n-3}$. In this case, $p_2g^s\in{_2\pi_2^s}\simeq\Z/2\{\eta^2\}$, i.e. $p_2g^s=\eta^2$.
In this case, the composition $f^s=\lambda^sg^s$ is detected by $h_1^2h_i$ which detects $\eta\eta_j\in{_2\pi_{2^j+1}^s}$, with $\eta_j$ being an element of Mahowald family, which is a decomposable term. That is $f^s=\eta\eta_j$ that consequently shows that $h(f)=0$, whenever $j\neq 2$. The case $j=2$ is also similar, showing that $f^s=\eta^2\nu=0$. This shows that it is not possible to have $f$ with this property and $h(f)\neq 0$. Therefore, $g^s$ pulls back to $P^{n-3}$.\\

\textbf{Case $i=3$.} We have $p_{n-3}g^s\in{_2\pi_3^s}\simeq\Z/2\{\nu\}$, hence detected by $h_2$ in the ASS. Consequently, $f^s$ is to be detected by $h_2h_j$ in the ASS. But, this implies that either $f^s=\eta\nu$ (if $j=1$) or $f^s=\nu^2$ (if $j=2$) or $f^s=\nu\sigma$ (if $j=3)$ (compare to the families listed in
\cite[Theorem 1.3]{CohenLinMahowald} and the six finite families of \cite[Page 3]{CohenLinMahowald}). If $h(f)\neq 0$ then $f=\nu^2$ showing that $f$ factors as $S^6\to P^3\to S^0$. Note that, stably $P^3\simeq P^2\vee S^3$ and $g^s$ can be constructed by being $\nu$ on $S^3$ and trivial in $P^2$. In fact in this case, we can take $\overline{\lambda|_{P^{3}}}:S^3\to S^0$ to be $\nu$. This verifies the conjecture in this case.\\

\textbf{Case $4,5$.} Suppose $g^s$ pulls back to $P^{n-4}$. Since ${_2\pi_4^s}\simeq 0$ as well as ${_2\pi_5^s}\simeq 0$, by a similar reasoning as above, $g^s$ factors through $P^{n-6}$ if it factors through $P^{n-4}$. \\

\textbf{Case $i=6$.} Suppose $g^s$ factors through $P^{n-6}$ but not $P^{n-7}$. In this case, $p_6g^s\in{_2\pi_6^s}\simeq\Z/2\{\nu^2\}$, so $p_6g^s$ is detected by $h_2^2$ in the ASS. According to \cite[Theorem 1.3]{CohenLinMahowald} we have a family which maps to $h_2^3$ in the Adams spectral sequence, showing that $f^s=\nu^3$. Consequently, $h(\nu^3)=0$. So, it is not possible to have $f$ with $h(f)\neq 0$. Hence, $g^s$ factors through $P^{n-7}$.\\

\textbf{Case $i=7$.} Suppose $g^s$ factors through $P^{n-7}$ but not $P^{n-8}$. In this case, $p_7g^s\in{_2\pi_7^s}\simeq\Z/16\{\sigma\}$. Since $h(g^s)\neq 0$, so $p_7g^s$ is an odd multiple of $\sigma$. We may multiply this by add an number, and without loss of generality, we may assume $p_7g^s=\sigma$. Consequently, and using computations of \cite[Families of Theorem 1.3 and Page 3]{CohenLinMahowald}, we conclude that $f^s$ is possibly detected by one of the following permanent cycles: $h_3^2$, $h_3h_4$, $h_3h_1$, $h_3h_2$. The fist one corresponds to $f=\theta_3=\sigma^2$. For the last two, $f^s=\sigma\eta$ or $f=\sigma\nu$ which by Theorem \ref{main0} we have $h(f)=0$. So, these cases cannot arise. If $f^s$ corresponds to $h_3h_4$ in the Adams spectral sequence, then $f^s\in{_2\pi_{22}^s}$ which is generated by $\nu\overline{\sigma}$ and $\eta^2\overline{\kappa}$ \cite[Table A3.3]{Ravenel-Greenbook}  which are decomposable elements and map trivially under $h$. Of course, for filtration reasons, it is not possible that $f\in\{\nu\overline{\sigma},\eta^2\overline{\kappa}\}$ as the generators of ${_2\pi_{22}^s}$ are detected in higher lines of the $E_2$-term of the ASS, and not the $2$-line.\\

\textbf{Case $i=8$.} Suppose $g^s$ through $P^{n-8}$ but not $P^{n-9}$. Then, $p_8g^s\in{_2\pi_8^s}\simeq\Z/2\{\eta\sigma,\varepsilon\}$ where $\varepsilon$ is represented by a Toda bracket $\langle\nu^2,2,\eta\rangle$. If $p_8g^s=\eta\sigma$ then, as above, we may use \cite{CohenLinMahowald} to show that $h(f)=0$. If $p_8g^s=\varepsilon$ then $g^s$ does not correspond to any of the families in the ASS for $P$ as $\varepsilon$ is detected by $c_0$ in the ASS. So, this case cannot arise at all. Consequently, we cannot have $f$ with $h(f)\neq $as above which does factor through $P^{n-8}$ but not $P^{n-9}$.\\

\textbf{Case $i=9$.} Suppose $g^s$ through $P^{n-9}$ but not $P^{n-10}$. Then, $p_9g^s\in{_2\pi_9^s}\simeq\Z/2\{\nu^3,\eta^2\sigma,\mu_9\}$ where $\mu_9$ is an element coming from ${_2\pi_9}J$ where $J$ is the fibre of the Adams operation $\psi^3-1:BSO\to BSO$.

\bibliographystyle{plain}

\end{document}